\documentclass[10pt, a4paper]{amsart}
\usepackage{a4wide}

\usepackage[utf8]{inputenc}

\usepackage{amsfonts, amstext, amsmath, amsthm, amscd, amssymb}
\usepackage{mathrsfs}
\usepackage{mathtools}
\usepackage{tikz}
\usepackage{url}
\usepackage[all]{xypic}

\usepackage{multirow,arydshln}

\usepackage{color}
\usepackage[colorlinks,
    linkcolor={red!50!black},
    citecolor={blue!50!black},
    urlcolor={blue!80!black}]{hyperref}

\usepackage{enumitem}

\renewcommand{\setminus}{{\smallsetminus}}

\newcommand{\bp}{\begin{pmatrix}}
\newcommand{\ep}{\end{pmatrix}}
\newcommand{\be}{\begin{equation}}
\newcommand{\ee}{\end{equation}}
\newcommand{\ol}[1]{\overline{#1}}

\newcommand{\smfrac}[2]{\mbox{\footnotesize$\displaystyle\frac{#1}{#2}$}} 

\numberwithin{equation}{section}

\theoremstyle{plain}
\newtheorem{theorem}[equation]{Theorem}
\newtheorem{lemma}[equation]{Lemma}
\newtheorem{proposition}[equation]{Proposition}

\newtheorem{corollary}[equation]{Corollary}

\newtheorem*{claim*}{Claim}

\theoremstyle{definition}

\newtheorem{definition}[equation]{Definition}
\newtheorem{construction}[equation]{Construction}

\theoremstyle{remark}

\newtheorem{remark}[equation]{Remark}
\newtheorem*{ack}{Acknowledgements}

\numberwithin{equation}{section}

\def\Z{\mathbb Z}
\def\R{\mathbb R}
\def\Q{\mathbb Q}

\def\wt#1{\widetilde{#1}}

\def\sm{\setminus}
\def\S{\Sigma}

\def\La{\Lambda}

\def\bp{\begin{pmatrix}}
\def\ep{\end{pmatrix}}
\def\ba{\begin{array}}
\def\ea{\end{array}}
\def\bn{\begin{enumerate}}
\def\en{\end{enumerate}}

\DeclareMathOperator\lk{lk}
\DeclareMathOperator\Wh{Wh}
\DeclareMathOperator\Ext{Ext}
\DeclareMathOperator\Tor{Tor}
\DeclareMathOperator\Hom{Hom}

\DeclareMathOperator\ord{ord}

\DeclareMathOperator\Bl{Bl}

\DeclareMathOperator\pt{pt}

\def\ga{{g_{\rm{alg}}}}
\def\gs{{g_\Z^{\rm{sh}}}}

\begin{document}

\title{The $\Z$-genus of boundary links}

\author{Peter Feller}
\address{ETH Zurich, R\"amistrasse 101, 8092 Zurich, Switzerland}
\email{peter.feller@math.ch}

\author{JungHwan Park}
\address{Georgia Institute of Technology, Atlanta, GA, USA}
\email{junghwan.park@math.gatech.edu }

\author{Mark Powell}
\address{Department of Mathematical Sciences, Durham University, UK}
\email{mark.a.powell@durham.ac.uk}



\begin{abstract}
The $\Z$-genus of a link $L$ in $S^3$ is the minimal genus of a locally flat, embedded, connected surface in $D^4$ whose boundary is $L$ and with the fundamental group of the complement infinite cyclic.
We characterise the $\Z$-genus of boundary links in terms of their single variable Blanchfield forms, and we present some applications. In particular, we show that a variant of the shake genus of a knot, the $\Z$-shake genus, equals the $\Z$-genus of the knot.
\end{abstract}

\maketitle

\section{Introduction}

A link in $S^3$ is an oriented 1-dimensional locally flat submanifold of $S^3$ homeomorphic to a nonempty disjoint union of 
circles. For a link $L$, let $X_L:= S^3 \sm \nu L$ be the link exterior and let $M_L$ be the result of $0$-framed surgery on $L$.
An $r$-component link $L = L_1 \cup \cdots \cup L_r$ in $S^3$ is a \emph{boundary link} if the components bound a collection of $r$ mutually disjoint Seifert surfaces in~$S^3$, or equivalently if there is an epimorphism $\pi_1(X_L) \twoheadrightarrow F_r$ onto the free group on $r$ generators, sending the oriented meridian of $L_i$ to the $i$th generator of~$F_r$.

Throughout the article we write $\Lambda:= \Z[t,t^{-1}]$ for the Laurent polynomial ring. Let $\phi \colon \Z^r \to \Z$ be the homomorphism that sends $e_i \to 1$ for each standard basis vector $e_i$ of $\Z^r$. Let $L$ be an {$r$-component} link with vanishing pairwise linking numbers. Use the compositions $\pi_1(X_L) \to H_1(X_L;\Z) \xrightarrow{\cong} \Z^r \xrightarrow{\phi}\Z$ and $\pi_1(M_L) \to H_1(M_L;\Z) \xrightarrow{\cong} \Z^r \xrightarrow{\phi}\Z$ to define the homology of $X_L$ and $M_L$ with $\La$ coefficients. Here the middle isomorphisms send each positive meridian to a different basis element $e_i$. The module $H_1(X_L;\La)$ is called the (single variable) \emph{Alexander module} of $L$.  For~$L$ a boundary link, as we show in Lemma~\ref{lem:XL-vs-ML} the Alexander module is canonically isomorphic to $H_1(M_L;\La)$. We will consider the Blanchfield pairing on $H_1(M_L;\La)$, which is technically simpler since $M_L$ is a closed $3$-manifold.

\begin{definition}
A \emph{$\Z$-surface} for a link $L$ is a locally flat, embedded, compact, oriented, connected surface in the $4$-ball $D^4$ whose boundary coincides with $L$ as oriented submanifolds of $S^3$.
The \emph{genus} of a connected surface $\Sigma$ with $m \geq 1$ boundary components is $\smfrac{1}{2}(2-\chi(\Sigma)-m)$. The \emph{$\Z$-genus} of a link $L$ is the minimal genus amongst all $\Z$-surfaces for $L$.
We say a link is \emph{$\Z$-weakly slice} if its $\Z$-genus is zero.
\end{definition}

We algebraically characterise the $\Z$-genus of boundary links, extending work of the first author with Lewark~\cite[Theorem 1.1]{FellerLewark_19} on the knot case.

\begin{theorem}\label{thm:main}
Let $L$ be an $r$-component boundary link and $M_L$ be the $0$-framed surgery on $L$. Then the following are equivalent.
\begin{enumerate}[font=\upshape]
  \item\label{item:main-1} The link $L$ bounds a $\Z$-surface of genus $g$.
  \item\label{item:main-2}
There is a size $2g$ Hermitian square matrix $A$ over $\La$ such that  $A(1)$ has signature $0$
and such that $A$ presents the Blanchfield form of $M_L$ on $TH_1(M_L;\La)$.
\item\label{item:main-3} There exists an embedding of the connected, oriented, genus $g$ surface with $(r+1)$ boundary components into $S^3$ such that $r$ of the boundary components coincide with $L$, and the final boundary component is a knot with Alexander polynomial~$1$.
\end{enumerate}
\end{theorem}

Theorem~\ref{thm:main} shows that for boundary links, having a genus $g$ $\Z$-surface implies the existence of a  genus $g$ $\Z$-surface given by a $\Z$-disc in $D^4$ union a collection of 2-dimensional 1-handles attached along the boundary and embedded in $S^3$. It was conjectured in~\cite{Feller-Lewark:2018-1} that this is the case for all links,
in other words that \eqref{item:main-1}$\Leftrightarrow$ \eqref{item:main-3} for all links.
Towards tackling this conjecture, we ask: what is the appropriate generalization of \eqref{item:main-2} that applies for all links?
For example, in order to define the coefficient system on $M_L$ we used that the pairwise linking numbers vanish, so the formulation of Theorem~\ref{thm:main} does not apply to all links.

For the proof of Theorem~\ref{thm:main}, \eqref{item:main-3} $\Rightarrow$ \eqref{item:main-1} is a consequence of the fact that Alexander polynomial $1$ knots are $\Z$-slice~\cite[Theorem~11.7B]{Freedman-Quinn:1990-1}. For~\eqref{item:main-2} $\Rightarrow$ \eqref{item:main-3}, the proof consists of reducing the statement for links to the statement for knots by performing internal band sums. Finally, \eqref{item:main-1} $\Rightarrow$ \eqref{item:main-2} is an algebraic topology computation, involving the intersection pairing of a suitably constructed $4$-manifold with boundary $M_L$ and fundamental group $\Z$.


Noting that $H_1(M_L;\La) \cong H_1(X_L;\La)$ for $L$ a boundary link, as we show in Lemma~\ref{lem:XL-vs-ML}, we deduce the following corollary, which is a natural generalisation of the aforementioned result that a knot is $\Z$-slice if and only if its Alexander module is trivial.

\begin{corollary}\label{cor:weakly-slice-torsion-free}
 A boundary link is $\Z$-weakly slice if and only if it has torsion-free Alexander~module.
 \end{corollary}

\subsection{Applications}\label{subsec:applications}

We describe several applications, whose proofs will be given in Section~\ref{sec:proofs-applications}. The first application exhibits a phenomenon for links with unknotted components. This uses the obstruction in Corollary~\ref{cor:weakly-slice-torsion-free}: we compute that the Alexander modules of the links in question have nontrivial torsion submodules.

\begin{figure}[h]\label{fig:Ln}
\includegraphics[width=.6\textwidth]{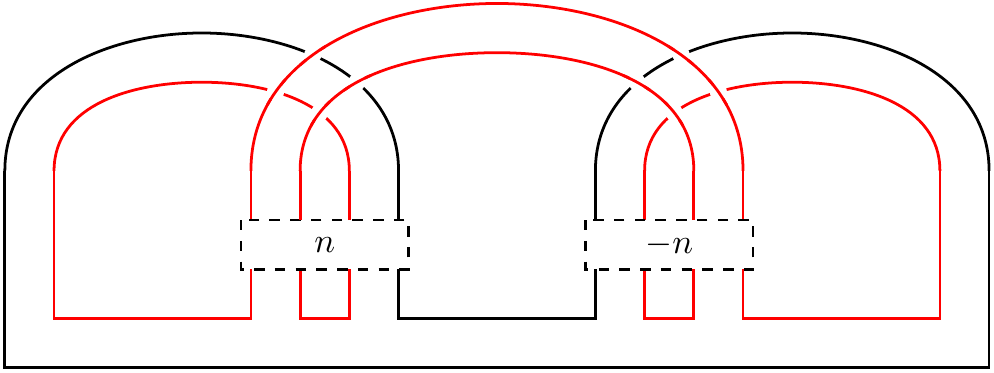}
\caption{The link $L_n$.  The dashed boxes indicate $\pm n$  full twists between the bands, without introducing any internal twisting to any of the bands.}
\end{figure}

\begin{corollary}
  The infinite family shown in Figure~\ref{fig:Ln} of $2$-component links $L_n$, where $n \neq 0,1$, are slice links, hence weakly slice, with unknotted components, but the $L_n$ are not $\Z$-weakly slice.
\end{corollary}

\begin{proof} Let $n  \neq 0,1$ be an integer
and let  $L_n$ be the link in Figure~\ref{fig:Ln}. It is not too hard to check that $L_n$ is a boundary link.  Hence we get the following computation using the obvious Seifert surface:
$$H_1(M_{L_n};\La)\cong H_1(X_{L_n};\La)\cong \La \oplus \La /\langle (n-1)t-n\rangle \oplus \La/\langle nt-(n-1)\rangle.$$
By Corollary~\ref{cor:weakly-slice-torsion-free}, it follows that $L_n$ is not $\Z$-weakly slice. The fact that $L_n$ is slice can be seen by performing a saddle move corresponding to a dual band to the middle band in Figure~\ref{fig:Ln}.
\end{proof}

Recall that a knot $K$ is \emph{shake slice} if the generator of second homology of the 4-manifold \[X_0(K):= D^4 \cup_{K \times D^2} D^2 \times D^2,\] obtained by attaching a 2-handle to the closed 4-ball $D^4$ along $K$ with framing zero,  known as the \emph{$0$-trace} of $K$, can be represented by a locally flat embedded 2-sphere $S \subseteq X_0(K)$.
Moreover we say that a knot is \emph{$\Z$-shake slice} if in addition $\pi_1(X_0(K) \sm S) \cong \Z$, generated by a meridian of $S$. This notion was introduced in~\cite{FMNOPR}.

Extending this, the \emph{$\Z$-shake genus} of a knot $K$ is the minimal genus of a surface $\Sigma$ representing a generator of $H_2(X_0(K);\Z)$ with $\pi_1(X_0(K) \sm \Sigma) \cong \Z$, again generated by a meridian of $\Sigma$.  Theorem~\ref{thm:main} enables us to characterise this quantity.

\begin{theorem}\label{cor:shake}
For all knots, the $\Z$-genus equals the $\Z$-shake genus.
\end{theorem}

The case $r=1$ of Theorem \ref{thm:main}, which is the main theorem of \cite{FellerLewark_19}, describes the $\Z$-genus of a knot algebraically, so this yields an algebraic characterisation of the $\Z$-shake genus. Note that we can also define the shake genus and the slice genus of a knot by dropping the condition on the fundamental groups. It is not known whether these two knot invariants differ in general. However, the shake genus and the slice genus do not coincide in the smooth category~\cite[Corollary 1.2]{Piccirillo:2019-1}.

Given a knot $K$, let $K_{a,b}$ denote the
$(a,b)$-cable of $K$, where $a$ is the longitudinal winding, and let $P_{p,n}(K)$ denote the link obtained by taking $p+n$ parallel copies of $K$, with pairwise vanishing linking numbers, where~$p$ components have the same orientation as $K$ and the remaining $n$-components have the opposite orientation.
As we show in Lemma~\ref{lem:shake}, the $\Z$-shake genus of $K$ can be reinterpreted as the minimal genus of a connected surface in $D^4$ with boundary $P_{\ell+1,\ell}(K)$, for some $\ell \geq 0$.  From this point of view, the next corollary extends Theorem~\ref{cor:shake}.  We compute the $\Z$-genus for $P_{p,n}(K)$ for every pair of nonnegative integers~$p$ and~$n$.

\begin{corollary}\label{cor:parallel}
Let $p$ and $n$ be nonnegative integers and let $w=p-n$. If $w=0$, then $P_{p,n}(K)$ is $\Z$-weakly slice, and otherwise
\[g_\Z(P_{p,n}(K))=g_\Z(K_{w,1})=g_\Z(K_{w,-1}) \leq g_\Z(K).\]
\end{corollary}

Theorem~\ref{cor:shake} can be recovered from the case $w=1$ and Lemma~\ref{lem:shake}.  
The fact that $g_\Z(K_{w,1}) \leq g_\Z(K)$ follows from \cite[Theorem 1.2]{Feller-Miller-Pinzon-Caicedo:2019-1} and \cite[Theorem~4]{McCoy:2019-1}.

Recall that an $r$-component link is a \emph{good boundary link} if there is a homomorphism $\theta \colon \pi_1(X_L) \to F_r$ sending the meridians to $r$ generators of the free group $F_r$, such that $\ker \theta$ is perfect; see \cite{Freedman:1993-1}, \cite{Freedman-Teichner:1995-2}, and~\cite{Cha-Kim-Powell} for more details.
An important open question related to topological surgery for 4-manifolds is whether every good boundary link is freely slice, that is bounds a disjoint union discs in $D^4$ such that the complement has free fundamental group~\cite[Corollary~12.3C]{Freedman-Quinn:1990-1}. We show that at least every good boundary link bounds a planar $\Z$-surface.

\begin{corollary}\label{cor:goodboundary}
Every good boundary link is $\Z$-weakly slice.
\end{corollary}

For a given link $L$, we can construct a new link called the \emph{Whitehead double}, denoted by $\Wh(L)$, by performing the untwisted Whitehead doubling operation on each component of $L$. Note that Whitehead doubling involves a choice of the sign for each clasp.  Recall that every Whitehead double of a link with vanishing linking numbers is a good boundary link, and hence by the previous corollary is $\Z$-weakly slice.

\begin{corollary}\label{cor:Wh}
If $L=L_1 \cup L_2$ is a $2$-component link, then for any choice of Whitehead double we have
\[g_\Z(\Wh(L))=\left\{
\begin{array}{cl}
0&\text{if $\lk(L_1,L_2)=0$},\\
1&\text{otherwise.}\end{array}\right.\]
Moreover, if $L=L_1 \cup L_2 \cup L_3$ is a $3$-component link, then $\Wh(L)$ is $\Z$-weakly slice if and only if either $($i$)$ $L$ has vanishing linking numbers, or $($ii$)$ for some $i,j,k$ with $\{i,j,k\}=\{1,2,3\}$:
\begin{enumerate}[label=(\emph{\alph*}), font=\upshape]
  \item the signs of the clasps of $\Wh(L_i)$ and $\Wh(L_j)$ disagree,
  \item $\lk(L_i,L_j)=0$,  and
  \item $|\lk(L_i,L_k)|=|\lk(L_j,L_k)|.$
\end{enumerate}
\end{corollary}

Let $K$ and $J$ be two oriented knots in $S^3$ that can be separated by an embedded $2$-sphere. Set $I=[0,1]$. Consider an embedded band $b\colon I \times I \to S^3$ with $b(I\times I)\cap K=b(I\times \{0\})$ and $b(I\times I)\cap J=b(I\times \{1\})$, where the orientations on these intervals coming from the orientations of the knots and from the intervals as a subset of the circle $\partial b(I\times I)$ agree.   We obtain a new knot \[K \#_b J \coloneqq K\cup J\cup b\left(\{0,1\}\times I\right)\setminus b\left((0,1)\times \{0,1\}\right)\] from \emph{band surgery} along $b$, which is called the the \emph{band connected sum} of $K$ and $J$ along $b$. If $b$ is trivial, that is if there exists an embedded $2$-sphere separating $K$ and $J$ such that the intersection of the sphere and the image of the band is an arc, then the band connected sum yields the connected sum $K\# J$. It was proven by Miyazaki~\cite[Theorem 1.1]{Miyazaki:1998-1} that there is a ribbon concordance from $K \#_b J$ to $K \# J$ for any band~$b$. In particular, $g_\Z(K \# J) \leq g_\Z(K \#_b J)$. This result can be thought of as follows.
Given a two component split link, which is a particular kind of boundary link, the connected sum of the components minimises the $\Z$-genus among all possible internal band sums on the link.  We extend this to all boundary links.

\begin{corollary}\label{cor:bandsum}
Let $L$ be an $r$-component boundary link and let $K_L$, $K_L'$ be knots, both of which are obtained by performing $r-1$ internal band sums on $L$. Furthermore, suppose that internal bands for $K_L$ are performed disjointly from some collection of disjoint Seifert surfaces for $L$. Then $$g_\Z(L) = g_\Z(K_L) \leq g_\Z(K_L').$$
\end{corollary}

\subsubsection*{Organisation}

Section~\ref{section:preliminaries} gives preliminaries on the Blanchfield form and Alexander duality in a disc, a useful generalisation of Alexander duality in a sphere.  Section~\ref{sec:2=>1} proves the implications \eqref{item:main-2} $\Rightarrow$ \eqref{item:main-3} $\Rightarrow$ \eqref{item:main-1}  in Theorem~\ref{thm:main}. Section \ref{sec:1=>2} proves that \eqref{item:main-1} implies \eqref{item:main-2}.
Section~\ref{sec:proofs-applications} proves the applications described in Section~\ref{subsec:applications}.

\begin{ack}
We thank Christopher Davis, Lukas Lewark, Duncan McCoy, Matthias Nagel, and Arunima Ray for helpful conversations and a couple of timely suggestions.  MP thanks the MPIM in Bonn, where he was a visitor during the period in which this paper was written. PF  gratefully acknowledges support by the SNSF Grant 181199.
\end{ack}

\section{Blanchfield forms and Alexander duality}\label{section:preliminaries}

\subsection{The Blanchfield form}

Let $M$ be a closed, oriented, connected 3-manifold equipped with a homomorphism $\pi_1(M) \to \Z$, giving rise to twisted homology and cohomology with coefficients in the $\La$-modules $\La$, $\Q(t)$, and $\Q(t)/\La$.
The \emph{Blanchfield form}~\cite{Blanchfield:1957-1} $\Bl_M$ is the nonsingular, sesquilinear, Hermitian form~\cite{Powell-QJM} defined on the torsion submodule $TH_1(M;\La)$ of $H_1(M;\La)$.
\[
\Bl_M \colon TH_1(M; \La) \times TH_1(M; \La) \to \Q(t)/\La,
\]
with adjoint $x\mapsto \Bl_M(-,x)$ given by the sequence of maps that we now describe; compare also~\cite{Borodzik-Friedl-Powell}. First, we use the Poincar\'{e} duality map $PD^{-1} \colon TH_1(M;\La) \xrightarrow{\cong} TH^2(M;\La)$.
The universal coefficient spectral sequence (see \cite[Section~2.1~and~2.4]{Hillman:2012-1-second-ed}) gives rise to an exact sequence as follows, where $\ol{\Ext}$ denotes that the involution on $\La$ determined by $t \mapsto t^{-1}$ has been used to alter the $\La$-module structure:
\[ 0  \to \ol{\Ext}^1_\La(H_1(M,\La),\La) \to  H^2(M;\La)\to \ol{\Ext}^0_\La(H_2(M,\La),\La).\]
Since $\ol{\Ext}^0_\La(H_2(M,\La),\La) = \Hom_\La(H_2(M,\La),\La)$ is torsion-free, we obtain a map $TH^2(M;\La)\to \ol{\Ext}^1_\La(H_1(M,\La),\La)$, which we then compose with the map
\[\ol{\Ext}^1_\La(H_1(M,\La),\La) \to \ol{\Ext}^1_\La(TH_1(M,\La),\La)\]
 induced by the inclusion from $TH_1(M;\La) \subseteq H_1(M;\La)$.
 Next, the Bockstein long exact sequence arising from the short exact sequence of coefficients $0 \to \La \to \Q(t) \to \Q(t)/\La \to 0$:
\[\xymatrix @R-0.8cm @C-0.2cm {
\ar[r] & \ol{\Ext}^0_{\La}(TH_1(M;\La),\Q(t)) \ar[r] & \ol{\Ext}^0_{\La}(TH_1(M;\La),\Q(t)/\La) \ar[r] & \\
\ar[r]& \ol{\Ext}^1_{\La}(TH_1(M;\La),\La) \ar[r] & \ol{\Ext}^1_{\La}(TH_1(M;\La),\Q(t)) \ar[r] & }\]
has first and last terms vanishing, the first since $TH_1(M;\La)$ is $\La$-torsion and the last since $\Q(t)$ is an injective $\La$-module.  Thus there is a map \[\ol{\Ext}^1_\La(TH_1(M;\La),\La)\to \ol{\Ext}^0_{\La}(TH_1(M;\La),\Q(t)/\La) = \Hom_\La(TH_1(M;\La),\Q(t)/\La).\]
The composition of these maps gives a homomorphism \[TH_1(M;\La) \to \Hom_\La(TH_1(M;\La),\Q(t)/\La),\] which as promised is the adjoint of the Blanchfield pairing.

\begin{definition}
  We say that the Blanchfield pairing is \emph{presented} by a Hermitian square matrix $A(t)$ over $\La$ of size $n$ if it is isometric to the pairing
\begin{align*}
 \ell_A \colon  \La^n/(A(t) \La^n) \times   \La^n/(A(t) \La^n) &\to \Q(t)/\La \\
 (v,w) & \mapsto v^T A(t^{-1})^{-1} \ol{w}.
\end{align*}
\end{definition}

For an $r$-component link $L$, we write $\Bl_{L}$ for the Blanchfield form of the zero surgery 3-manifold~$M_L$, defined using the homomorphism $\pi_1(M_L) \to \Z$ sending each oriented meridian to $1 \in \Z$.

\subsection{Alexander duality in a disc}\label{sec:Alexander-duality}

In this section we briefly recall a version of Alexander duality for the disc.
Let $X$ be a submanifold properly embedded in a disc $D^n$ and assume that $X$ admits an open tubular neighbourhood $\nu X$.

\begin{proposition}\label{proposition:alex-duality-disc}
For every $k \in \Z$, we have:
\[H_k(D^n \sm \nu X) \cong \wt{H}^{n-k-1}(S^{n-1} \cup_{\nu \partial X} \nu X).\]
\end{proposition}

\begin{proof}
 We have
\begin{align*}
  H_k(D^n \sm \nu X) &\cong H^{n-k}(D^{n} \sm \nu X, (S^{n-1}\sm \nu\partial X) \cup (\partial \nu X \sm \nu \partial X)) \cong H^{n-k}(D^n, S^{n-1} \cup \nu X)
\end{align*}
by the composition of Poincar\'{e}-Lefschetz duality, and excision.
We consider the long exact sequence of the pair:
\[H^{n-k-1}(D^n) \to H^{n-k-1}(S^{n-1} \cup \nu X) \to H^{n-k}(D^n,S^{n-1} \cup \nu X) \to H^{n-k}(D^n).\]
Thus, $H^{n-k-1}(S^{n-1} \cup \nu X) \cong H^{n-k}(D^n,S^{n-1} \cup \nu X)$ unless $k=n-1,n$. For $k=n$ both the left and right hand sides in the statement of the proposition vanish: $H_n(D^n \sm \nu X) =0= \wt{H}^{-1}(S^{n-1} \cup_{\nu \partial X} \nu X)$.  In the case that $k=n-1$, we obtain a short exact sequence
\[0 \to \Z \to H^{0}(S^{n-1} \cup \nu X) \to H^{1}(D^n,S^{n-1} \cup \nu X) \to 0, \]
so $\wt{H}^0(S^{n-1} \cup \nu X) \cong H^{1}(D^n,S^{n-1} \cup \nu X) \cong H_{n-1}(D^n \sm \nu X)$.
Therefore, we obtain our statement for Alexander duality in a disc as claimed:
\[H_k(D^n \sm \nu X) \cong \wt{H}^{n-k-1}(S^{n-1} \cup_{\nu \partial X} \nu X).\qedhere\]
\end{proof}

\section{\eqref{item:main-2} implies \eqref{item:main-3} implies \eqref{item:main-1}}\label{sec:2=>1}

We prove two of the implications in Theorem~\ref{thm:main} in this section.
We need a purely algebraic lemma that shows up in this section and in Section~\ref{sec:1=>2}.

\begin{lemma}\label{lemma:algebraic}
  Let $g \geq 0$ and $r \geq 1$ be integers, and let $Q$ be a $\La$-module with a presentation
  \[\La^{r-1+2g} \xrightarrow{B} \La^{r-1+2g} \to Q \to 0\]
  where $B$  is a square matrix over $\La$ of the form
  \[B = \begin{pmatrix}
    0_{(r-1) \times (r-1)} & 0_{(r-1) \times 2g} \\ 0_{2g \times (r-1)} & A_{2g \times 2g}
  \end{pmatrix}\]
  with $\det(A) \neq 0$. Then the following holds.
\begin{enumerate}[label=(\emph{\roman*}),font=\upshape]
  \item\label{item:algebraic-i} The torsion submodule of $Q$, denoted by $TQ$, is presented by $\La^{2g} \xrightarrow{A} \La^{2g} \to TQ \to 0$. In particular, $\ord TQ = \det A$.
  \item\label{item:algebraic-iii} The $\Lambda$-module $Q$ decomposes as $TQ \oplus \La^{r-1}$.
\end{enumerate}
\end{lemma}

\begin{proof}
  Let $\{e_1,\dots,e_{r-1+2g}\}$ be the standard basis of $\La^{r-1+2g}$, and write $[e_i]$ for the image of $e_i$ under the quotient $\La^{r-1+2g} \to Q$.
  Consider the subgroup
  \[T:= \langle [e_{r}],\dots,[e_{2g+1-1}]\rangle\]
  generated by the shown subset of the $[e_i]$.  Each $[e_{i}]$ for $i= r,\dots,r-1+2g$ is $\det(A)$-torsion; in particular, $T$ consists entirely of torsion, hence $T\subseteq TQ$.
On the other hand, given the form of $B$, the remaining generators $\{[e_i]\}_{i=1}^{r-1}$ generate a free summand $\La^{r-1}$, and it is straightforward to see that $Q= TQ \oplus \La^{r-1}$.
\end{proof}

We fix the following setup. Let $L$ be an $r$-component boundary link and let $\{F_i\}_{i=1}^{r}$ be a collection of disjoint Seifert surfaces for $L$. Tube the surfaces $\{F_i\}_{i=1}^{r}$ together along $r-1$ tubes disjoint from the surfaces to obtain a connected Seifert surface $F$, of genus $g$ say.  Let $N$ be a Seifert matrix for $L$ of size $r-1+2g$ obtained by picking a basis $\{\gamma_1,\ldots,\gamma_{r-1+2g}\}$ of $H_1(F;\Z)$ of some Seifert surface for $L$ as follows: the first $r-1$ elements are given by meridians for the tubes used in the construction of $F$, while the next $2g$ elements are given by simple, oriented, closed curves disjoint from the meridians of the tubes that form a symplectic basis of the closed surface $F/\{L_i\}_{i=1}^r$ given by crushing each component of $L$ to a distinct point.  Let $V$ denote the Seifert matrix representing the Seifert form restricted to span of the last $2g$ basis elements.
In particular, we have that $\det(V-V^T)=1$.  We have that $N$ has the form
\begin{equation}\label{eqn:Seifert-form}
  N = \begin{pmatrix}
    0_{(r-1) \times (r-1)} & 0_{(r-1) \times 2g} \\ 0_{2g \times (r-1)} & V_{2g \times 2g}
  \end{pmatrix}
\end{equation}
since meridians to the tubes link themselves and all other curves trivially.

Let $K_L$ be a separating curve on $F$ such that one of the components of $F\setminus K_L$ contains $\partial F$, while the other contains the simple closed curves representing $\gamma_i$ for $r\leq i\leq r-1+2g$.
{We can take $K_L$ to be a collection of push-offs of the components of $L$, banded together along the tubes used in the construction of $F$. Hence $K_L$ is a knot obtained by performing $r-1$ internal band sums on $L$, where internal bands are disjoint from $\{F_i\}_{i=1}^{r}$.}
Then by construction $K_L$ has $V$ as a Seifert matrix.

For $i=1,\dots,r-1+2g$, let $e_i \in H_1(S^3 \sm F;\Z)$ be the element that is Alexander dual to $\gamma_i \in H_1(F;\Z)$.  We also write $e_i$ for $e_i \otimes 1 \in H_1(S^3 \sm  F;\Z) \otimes_{\Z} \La \cong \La^{r-1+2g}$.  Recall that $\{e_i\}_{i=1}^{r-1+2g}$ generates the Alexander module $H_1(X_L;\La)$; see e.g.~\cite[Theorem~6.5]{Lickorish:1997-1}. 

\begin{lemma}\label{lem:algboundarylinkchar}
Let $L$ be an $r$-component link boundary link. Then the following holds.
\begin{enumerate}[label=(\emph{\roman*}),font=\upshape]
  \item\label{item:algebraic--ii} Let $T$ be the $\La$-torsion submodule of $H_1(X_L;\La)$. Then $H_1(X_L;\La) \cong \Lambda^{r-1}\oplus T$ and $\ord(T)(1)= \pm 1$.
  \item\label{item:algebraic--iii} For any choice of Seifert surface $F$ and basis $\{\gamma_i\}_{i=1}^{r-1+2g}$ as above, giving rise to an identification
  \[H_1(X_L;\La)=\Lambda^{r-1+2g}/((tN-N^T)\Lambda^{r-1+2g})\] generated by $\{[e_i]\}_{i=1}^{r-1+2g}$, the torsion submodule $T$ of $H_1(X_L;\La)$ is spanned by $[e_{r}],\dots,[e_{r-1+2g}]$, and presented by $tV-V^T$.
\end{enumerate}
\end{lemma}

\begin{proof}
Identify $H_1(X_L;\La)=\Lambda^{r-1+2g}/((tN-N^T)\Lambda^{r-1+2g})$.
Note that $tN-N^T$ has the form
\[\begin{pmatrix}
    0_{(r-1) \times (r-1)} & 0_{(r-1) \times 2g} \\ 0_{2g \times (r-1)} & A_{2g \times 2g}
  \end{pmatrix},\]
where $A = tV-V^T$, by \eqref{eqn:Seifert-form}.   Since $\det(A(1)) = \det(V-V^T) =1$, Lemma~\ref{lemma:algebraic} applies with $Q=H_1(X_L;\La)$. We deduce that $H_1(X_L;\La) \cong \Lambda^{r-1}\oplus T$, with $T=TH_1(X_L;\La)$ spanned by $[e_{r}],\dots,[e_{r-1+2g}]$ and presented by $A = tV-V^T$. Then the order of $T$ at $t=1$ is $\ord (T)(1) = \det(A)(1) = \det(A(1)) = \det(V-V^T)=\pm 1$.
\end{proof}

\begin{lemma}\label{lem:XL-vs-ML}
For an $r$-component boundary link $L$, $H_1(X_L;\La) \cong H_1(M_L;\La)$.
\end{lemma}

\begin{proof}
The zero framed longitudes of the components of $L$ determine elements $\{\ell_1,\dots,\ell_r\}$ in $H_1(X_L;\La)$ since the linking numbers are zero.
It follows from a straightforward Mayer-Vietoris computation that the homology of the zero surgery is the quotient $H_1(X_L;\La)/\langle \ell_1,\dots,\ell_r\rangle$.  Since the longitudes bound disjoint Seifert surfaces, they live in the second derived subgroup of $\pi_1(X_L)$, and are therefore trivial in $H_1(X_L;\La)$. It follows that $H_1(X_L;\La) \cong H_1(M_L;\La)$ as desired.
\end{proof}

With Lemma~\ref{lem:XL-vs-ML} in mind, we will therefore be working with the Blanchfield form $\Bl_L$ on the closed 3-manifold $M_L$.  Note that a knot $K_L$ is also evidently a boundary link, so $H_1(X_{K_L};\La) \cong H_1(M_{K_L};\La)$ and we write $\Bl_{K_L}$ for the Blanchfield form on $M_{K_L}$.

\begin{theorem}\label{thm:2=>1}
Let $L$ be a boundary link.
\begin{enumerate}[label=(\emph{\roman*}),font=\upshape]
  \item\label{item:2=>1 i} The Blanchfield form on the torsion $TH_1(M_L;\La)$
is isometric to the Blanchfield form on the Alexander module $H_1(M_{K_L};\La)$ of $K_L$.
 \item\label{item:2=>1 ii} If there is a size $2g$ Hermitian square matrix $A$ over $\La$
such that $A$ presents the Blanchfield form of $M_L$ on $TH_1(M_L;\La)$ and $A(1)$ has signature $0$,
then $L$ cobounds an embedded cobordism in $S^3$, of genus $g$, with an Alexander polynomial $1$ knot.
\end{enumerate}
\end{theorem}

The second item proves the implication Theorem~\ref{thm:main}~\eqref{item:main-2} $\Rightarrow$ \eqref{item:main-3}.

Let $S \subseteq \La$ denote the smallest multiplicative subset containing $t-1$.  We write $\La_S :=  \Z[t,t^{-1},(t-1)^{-1}]$ for the ring obtained from $\La$ by inverting the elements in $S$.  This is a commutative and therefore flat localisation.

\begin{proof} We prove~\ref{item:2=>1 i} first.
Define $H_L:= (1-t)N+(1-t^{-1})N^T$ and $H_{K_L}:=(1-t)V+(1-t^{-1})V^T$.
By~\cite[Theorem 1.1]{Conway}, we have an isomorphism
\[\phi\colon H_1(M_L;\La) \otimes_\La \Lambda_S\to \Lambda_S^{r-1+2g}/H_L(t)^T\Lambda_S^{r-1+2g}\]
 such that $\phi$ induces an isometry between $\Bl_{L}\otimes \La_S$ and the linking form
\[\lambda_{H_L}:=\Tor(\Lambda_S^{r-1+2g}/H_L^T\Lambda_S^{r-1+2g})\times \Tor(\Lambda_S^{r-1+2g}/H_L^T\Lambda_S^{r-1+2g})\to \Q(t)/\La_S,\] where
\[\lambda_{H_L}([v],[w])=-\frac{1}{\Delta^2}v^TH_L(t)\overline{w}\]
 for $\Delta=t^{-g}\det(tV-V^T)=\Delta_{K_L}\in\La$ the order of $TH_1(M_L;\La)$. However, by Lemma~\ref{lem:algboundarylinkchar}, there is an isometry between the abstract pairings $\lambda_{H_L}$ and $\lambda_{H_{K_L}}$, since both correspond to $tV-V^T$.

 We know that $\lambda_{H_L}$ corresponds to the Blanchfield pairing on $M_L$ over $\La_S$.
 Since $\lambda_{H_{K_L}}$ is isometric to $\Bl(K_L)\otimes \La_S$ (again by~\cite[Theorem 1.1]{Conway}), we conclude that  $\Bl_{L}\otimes \La_S$ is isometric to $\Bl_{K_L}\otimes \La_S$.
 We therefore have an isometry between  $\Bl(K_L)\otimes \La_S$ and $\Bl(K_L)\otimes \La_S$.
 However, by~\cite[Proposition~1.2]{Levine-77-knot-modules} multiplication by $t-1$ induces an isomorphism on \[H_1(X_{K_L};\La) \cong TH_1(M_{K_L};\La) \cong TH_1(M_L;\La).\]
 We therefore have that $\Bl_{L}\otimes \La_S$ is isometric to $\Bl_{K_L}\otimes \La_S$ if and only if $\Bl_{L}$ is isometric to $\Bl_{K_L}$; see e.g.~\cite[Proposition~A.2]{FellerLewark_19}.  So indeed $\Bl_{L}$ is isometric to $\Bl_{K_L}$, completing the proof of~\ref{item:2=>1 i}.

For~\ref{item:2=>1 ii}, we use that $K_L$ is cobordant via a genus $g$ cobordism $\Sigma$ in $S^3$ to an Alexander polynomial $1$ knot if and only if there exists a Hermitian matrix $A$ of size $g$, with signature of $A(1)$ zero, that presents the Blanchfield form of $K_L$~\cite[Theorem 1.1]{FellerLewark_19}.
Here, a cobordism $\Sigma$ in $S^3$ between links $L_1$ and $L_2$ is an oriented surface with boundary a link that consist of the disjoint union of the two links $L_1$ and $L_2^{\operatorname{rev}}$, {where $L_2^{\operatorname{rev}}$ is the link obtained by reversing the orientation of each component of $L_2$.}
By hypothesis and \ref{item:2=>1 i} such a Hermitian matrix exists.
 Moreover, this cobordism may be assumed to be constructed from $K_L \times I$ union $2g$ 2-dimensional $1$-handle attachments to $K_L \times \{1\}$.
We may and shall choose $\Sigma$ such that $K_L \times I$ induces the $0$-framing of $K_L \times \{0\}$, i.e.\ extends to a Seifert surface for $K_L$. Compare~\cite[Lemma 18]{Feller-Lewark:2018-1}.

Let $C$ be the connected cobordism in $S^3$ of genus $0$ from $L$ to $K_L$ given by the component of $F\setminus K_L$ containing $\partial F$.  By general position, we may and shall assume that the $1$-handles of $\Sigma$ are disjoint from $C$. Then, using that both $K_L \times I$ and $C$ induce the 0-framing on $K_L$, by further isotopy arrange that $(K_L \times I) \cap C =K_L$. Therefore, the union $C \cup_{K_L} \Sigma$ is the embedded cobordism we seek between $L$ and the Alexander polynomial $1$ knot.
\end{proof}

The proof of Theorem~\ref{thm:main}~\eqref{item:main-3} $\Rightarrow$ \eqref{item:main-1} is by
a standard argument; compare e.g.~\cite{Rudolph_84_SomeTopLocFlatSurf,Feller-2015,FellerLewark_19}.

\begin{proof}[Proof of Theorem~\ref{thm:main}~\eqref{item:main-3} $\Rightarrow$ \eqref{item:main-1}
]
Glue together:
 \begin{itemize}
   \item The hypothesised connected cobordism $C$ of genus $g$ from $L$ to an Alexander polynomial $1$ knot $J$, pushed into $S^3 \times I$ so that $L = C \cap (S^3 \times \{0\})$ and $J = C \cap (S^3 \times\{1\})$;
   \item A $\Z$-disc $D$ in $D^4$ for the Alexander polynomial $1$ knot $J$.
 \end{itemize}
 This yields a genus $g$ surface $S:= C \cup_{J} D \subseteq D^4 = (S^3 \times I) \cup D^4$ with boundary $L$.
Since $C$ is obtained from pushing a surface in $S^3$ into $S^3 \times I$, we may assume that it is obtained from $J$ by band moves. Hence the exterior of $C$ can be built from $S^3 \sm \nu J$ by attaching 4-dimensional 2-handles to $(S^3 \sm \nu J) \times I$. Hence $\pi_1(S^3 \sm \nu J) \to \pi_1(S^3 \times I \sm \nu C)$ is surjective. Since $\pi_1(D^4 \sm \nu D) \cong \Z$ and $\pi_1(S^3 \sm \nu J)$ are both normally generated by the meridian of $J$, it follows from the Seifert-Van Kampen theorem that $\pi_1(D^4 \sm\nu S) \cong \Z$, so that $S$ is the desired $\Z$-surface of genus~$g$ for~$L$.
\end{proof}

\section{\eqref{item:main-1} implies \eqref{item:main-2}}\label{sec:1=>2}

Let $L$ be an ordered, oriented, $r$-component link.  Write $X_L := S^3 \sm \nu L$ for the exterior of the link.  As above, we use the representation $\pi_1(X_L) \to \Z$ defined by $\phi \colon \pi_1(X_L) \to H_1(X_L;\Z) \cong \Z^{r} \to \Z$ given by concatenating the abelianisation homomorphism, the identification with $\Z^r$ given by sending the $i$th ordered, oriented meridian to $e_i$, and the map $\sum_{i=1}^r a_i e_i \mapsto \sum_{i=1}^r a_i$.  This is sometimes  called the \emph{total linking number} representation~\cite[Section~2.5]{Hillman:2012-1-second-ed}.
Note that this representation is independent of the ordering of $L$, so is well-defined for unordered links. In this section we show that \eqref{item:main-1} implies \eqref{item:main-2} in Theorem~\ref{thm:main}.
As always, $\La := \Z[\Z]$.

We will prove this implication for a slight generalisation of boundary links, in order to make clear precisely which properties we are using in the proof.
Note that all the links we consider will in particular have pairwise linking numbers vanishing, so that the coefficient system $\phi$ extends to the zero-surgery manifold~$M_L$.

\begin{definition}\label{defn:ZTS}
  We say that an $r$-component link $L$ in $S^3$ with pairwise linking numbers zero has a \emph{$\Z$-trivial surface system} if there is a collection of Seifert surfaces $\{F_i\}_{i=1}^{r-1}$ for all but one of the components of $L$, each of whose interiors is embedded in $S^3 \sm L$ (the surfaces may intersect one another), such that for every $i$, for every simple closed curve $\gamma$ on $F_i$, and for every basing of $\gamma$, we have that $\phi(\gamma) =0 \in \Z$.  We refer to a link that admits a $\Z$-trivial system of surfaces as a \emph{$\Z$TS-link}.
\end{definition}

\begin{lemma}\label{lemma:ZTS-implies-bdy}
  Every boundary link is a $\Z$TS link.
\end{lemma}

\begin{proof}
  Curves in the interior of a boundary link Seifert surface are trivial in $H_1(S^3 \sm L;\Z)$.
\end{proof}

We will prove the following result in this section, which combined with Lemma~\ref{lemma:ZTS-implies-bdy} implies that \eqref{item:main-1} implies \eqref{item:main-2} in Theorem~\ref{thm:main}.

\begin{theorem}\label{thm:weakly-Z-slice-implies-Alexander-module-condition-bdy-links}
Let $L$ be an $r$-component $\Z$TS-link that bounds a connected $\Z$-surface of genus $g$ in $D^4$. Let $M_L$ be the $0$-framed surgery on $L$.
Then   \[H_1(M_L;\La) \cong \La^{r-1} \oplus TH_1(M_L;\La)\]
 and $\ord(TH_1(M_L;\La))(1) \doteq 1$.
Moreover there is a size $2g$ Hermitian square matrix $A$ over $\La$ such that $A$ presents the Blanchfield form of $M_L$ on $TH_1(M_L;\La)$ and $A(1)$ has signature $0$.
\end{theorem}

Let $P \subseteq D^4$ be the hypothesised connected, compact, oriented surface, locally flat embedded into $D^4$.
We note the following about the homology of $D^4 \sm \nu P$.

\begin{lemma}\label{lemma:hom-D4-minus-P}
	The nonvanishing homology groups of $D^4 \sm \nu P$ are as follows.
	\[H_i(D^4\sm \nu P;\Z) \cong \begin{cases} \Z & i=0,1 \\ \Z^{2g+ r-1} & i=2. \end{cases}\]
\end{lemma}

\begin{proof}
	This follows from Alexander duality in a disc, Proposition~\ref{proposition:alex-duality-disc}, with $X=P$ and $n=4$.
\end{proof}

\begin{remark}\label{remark:what-we-use}
We shall not assume that $\pi_1(D^4 \sm \nu P)\cong \Z$. Instead we will assume $H_1(D^4 \sm \nu P;\La)=0$, or equivalently that $\pi_1(D^4 \sm \nu P)$ has perfect commutator subgroup.   Similar remarks apply to the manifold $W_P$ constructed below.  This generalisation will be useful later in the proof of Lemma~\ref{lem:shake}, and helps to  clarify the proof.\end{remark}

As in Remark~\ref{remark:what-we-use}, we assume that there is a short exact sequence $1 \to \Gamma \to \pi_1(D^4 \sm\nu P) \to \Z \to 0$ with the surjection equal to the abelianisation, with the commutator subgroup $\Gamma$ a perfect group. This implies that $\pi_1(D^4 \sm \nu P) \cong \Z \ltimes \Gamma$, with the $\Z$ action determined by conjugation, and it implies that $H_1(D^4 \sm \nu P;\La) =0$.
Here we use the abelianisation $\pi_1(D^4 \sm \nu P) \to \Z$ to extend the~$\La$ coefficient system.

Towards understanding the twisted homology, we start with the following general computation.

\begin{lemma}\label{lemma:general-homology}
  Let $W$ be a compact, connected, oriented, topological $4$-manifold with $\pi_1(W) \cong \Z$, and suppose that $\partial W$ is nonempty, connected, oriented, and that $\pi_1(\partial W) \to \pi_1(W)$ is onto.  Then
  \[H_i(W;\La) \cong \begin{cases}
    \Z & i=0 \\
    \La^{\beta_2(W)} & i=2 \\
    0 & \text{otherwise.}
  \end{cases}\]
  The same holds if instead of $\pi_1(W) \cong \Z$ we assume $H_1(W;\La)=0$.
\end{lemma}

\begin{proof}
  Note that $H_i(W;\La) \cong H_i\big(\wt{W};\Z\big)$ for all $i$, where $\wt{W}$ is the universal cover.  Since $W$ is connected we have that $H_0\big(\wt{W};\Z\big) \cong \Z$ and $H_1\big(\wt{W};\Z\big) =0$.  Next we show that $H_2(W;\La)$ is free.  $H_2(W;\La) \cong H^2(W,\partial W;\La)$ by Poincar\'{e}-Lefschetz duality.  The universal coefficient spectral sequence (UCSS) computing cohomology in terms of homology has $E_2$-page
  \[E_2^{p,q} = \Ext^q_{\La}(H_p(W,\partial W;\La),\La)\]
and converges to the cohomology $H^{*}(W,\partial W;\La)$.  Since $\partial W$ is connected and $\pi_1(\partial W) \to \pi_1(W)$ is onto, it follows that $H_0(\partial W;\La) \cong \Z$.  Therefore the long exact sequence of the pair
\[0=H_1(W;\La) \to H_1(W,\partial W;\La) \to H_0(\partial W;\La) \cong \Z \xrightarrow{\cong} H_0(W;\La) \cong \Z \to H_0(W,\partial W;\La) \to 0\]
  implies that $H_i(W,\partial W;\La)=0$ for $i=0,1$. It follows that the $p=0$ and $p=1$ columns of the $E_2$ page of the UCSS vanish, and thus the remaining nonzero term $\Ext^0_{\La}(H_2(W,\partial W;\La),\La)$ on the $2$-line equals $H^2(W,\partial W;\La)$.  Note that the vanishing of the $p=0,1$ columns also precludes the possibility of any differentials influencing this outcome.
  By~\cite[Lemma~2.1]{Borodzik-Friedl-2013-alg-unknotting}, $\Ext^0_{\La}(H,\La)$ is a free $\La$-module for every $\La$-module $H$.   So $H_2(W;\La)$ is a free $\La$-module as claimed.  We will compute its rank later.

  Now $H_3(W;\La) \cong H^1(W,\partial W;\La)$. Again $H_i(W,\partial W;\La)=0$ for $i=0,1$, fed into the UCSS, implies that $H^1(W,\partial W;\La) =0$.

  To complete the computation of the homology it remains to compute the rank of $H_2(W;\La)$, in other words the dimension of $H_2(W;\La) \otimes_{\La} \Q(t) \cong H_2(W;\Q(t))$. First we show that this equals the Euler characteristic $\chi(W)$.  We used above that $\Q(t)$ is flat as a $\La$-module.  This also implies that $H_i(W;\Q(t)) \cong H_i(W;\La) \otimes_{\La} \Q(t) =0$ for $i \neq 2$. Therefore $\chi(W) = \dim H_2(W;\Q(t))$, and so the rank of $H_2(W;\La)$ equals $\chi(W)$ as asserted.

  It remains to compute the Euler characteristic of $W$ by computing its rational homology.  First $H_0(W;\Q) \cong \Q \cong H_1(W;\Q)$.  We have $H_3(W;\Q) \cong H^1(W,\partial W;\Q) \cong H_1(W,\partial W;\Q)$ by Poincar\'{e}-Lefschetz and universal coefficients.  Then the long exact sequence:
  \[H_1(\partial W;\Q) \twoheadrightarrow H_1(W;\Q) \cong \Q \to H_1(W,\partial W;\Q) \to H_0(\partial W;\Q) \cong \Q \xrightarrow{\cong} H_0(W;\Q) \cong \Q\]
  implies that $H_1(W,\partial W;\Q)=0$.  It follows that \[\chi(W) = \beta_2(W) -\beta_1(W) + \beta_0(W) = \beta_2(W) -1+1 = \beta_2(W).\]
  So in fact the rank of $H_2(W;\La)$ is $\beta_2(W)$. This completes the proof of the lemma.
\end{proof}

Let $F:= P \cup_{\partial P} \bigcup_{i=1}^r D^2$, a closed surface of genus $g$. Let $G$ be a handlebody with $\partial G=F$.
Define
\[W_P := D^4 \sm \nu P \cup_{P \times S^1} (G \times S^1).\]
  For this gluing we must choose some diffeomorphism of $P \times S^1 \subseteq F \times S^1$ relative to the boundary of $P$. There are self-diffeomorphisms of $P \times S^1$ corresponding to changes in framing for the trivial bundle $\nu P \cong P \times D^2$. We choose the framing to satisfy that curves of the form $\gamma_k \times \{\pt\}$, where $\gamma_k \subseteq P$ is a simple closed curve, lie in the commutator subgroup of $\pi_1(D^4 \sm \nu P)$, and use this to define the gluing. In the case that  $\pi_1(D^4 \sm \nu P) \cong \Z$, this means that the curves $\gamma_k \times \{\pt\}$ are null-homotopic in $D^4 \sm \nu P$.  Note that $\partial W_P = M_L$, the zero surgery on $L$.

\begin{construction}\label{construction:spheres-D4-sm-P}
We construct a collection of surfaces in $D^4 \sm \nu P$.
Let $\{\gamma_k\}_{k=1}^{r-1+2g}$ be a basis for $H_1(P;\Z)$, consisting of $r-1$ curves parallel to $r-1$ of the components of $L$, and a symplectic collection of $2g$ curves disjoint from those.   Consider their push-offs $\{\gamma_k \times \{\pt\}\}_{k=1}^{r-1+2g}$ in $P \times S^1$.  Each $\gamma_k$ lies in the (perfect) commutator subgroup of $\pi_1(D^4 \sm \nu P)$ by our choice of framing of the normal bundle of $P$ made above. For each $k$ let $D_k \looparrowright D^4 \sm \nu P$ be an immersed $\Z$-trivial surface with boundary $\gamma_k \times \{\pt\}$. That is the induced map $\pi_1(D_k) \to \pi_1(D^4 \sm \nu P) \xrightarrow{\phi} \Z$ is the trivial homomorphism.
Use $D_k$ to surger the torus $\gamma_k \times S^1$ to an immersed surface, that we call $\S_k$.
  \end{construction}

Recall that we write $\La_S :=  \Z[t,t^{-1},(t-1)^{-1}]$.

\begin{lemma}\label{lemma:H2-and-basis-D4-P}
  The nonvanishing homology groups of $D^4 \sm \nu P$ are as follows.
	\[H_i(D^4\sm \nu P;\La) \cong \begin{cases} \Z & i=0 \\ \La^{2g+ r-1} & i=2. \end{cases}\]
A basis for $H_2(D^4\sm \nu P;\La_S)$ is given by the collection of immersed surfaces $\{\S_k\}_{k=1}^{r-1+2g}$.
\end{lemma}

\begin{proof}
  By Lemma~\ref{lemma:hom-D4-minus-P}, $\beta_2(D^4 \sm \nu P) = r-1+2g$.  The computation of the homology groups with $\La$ coefficients then follows from Lemma~\ref{lemma:general-homology}, noting that $W=D^4 \sm \nu P$ indeed satisfies the hypotheses of that lemma.

We show that the $\{\S_k\}$ are a basis over $\La_S$.
To do this we consider the exact sequence
\begin{equation}\label{eqn:PxS1}
H_2(P \times S^1;\La) \to H_2(D^4 \sm \nu P;\La) \to H_2(D^4 \sm \nu P,P \times S^1;\La) \to H_1(P \times S^1;\La) \to 0.
\end{equation}
We know that \[H_i(P \times S^1;\La) \cong H_i(P \times \R;\Z) \cong H_i(P;\Z)\]
for $i=1,2$. For $i=1$ we therefore have $H_i(P \times S^1;\La)\cong \Z^{r-1+2g}$ generated by the $\{\gamma_k\times \{\pt\}\}_{k=1}^{r-1+2g}$, while for $i=2$ we have $H_2(P\times S^1;\La) \cong H_2(P;\Z) =0$.
Therefore $H_i(P \times S^1;\La_S) =0$ for $i=1,2$, so over $\La_S$
\begin{equation}\label{eqn:La_s-iso}
  H_2(D^4 \sm \nu P;\La_S) \xrightarrow{\cong} H_2(D^4 \sm \nu P,P \times S^1;\La_S)
\end{equation}
is an isomorphism.

Since $H_2(D^4 \sm \nu P;\La)\cong \La^{r-1+2g}$ it follows that $H_2(D^4 \sm \nu P,P \times S^1;\La_S) \cong \La_S^{r-1+2g}$.  We claim the following.

\begin{claim*}
 We have that $H_2(D^4 \sm \nu P,P \times S^1;\La) \cong \La^{r-1+2g}$ is a free module of the same rank as $H_2(D^4 \sm \nu P;\La)$.
\end{claim*}

  To see this note that $H_2(D^4 \sm \nu P,P \times S^1;\La) \cong H^2(D^4 \sm \nu P,S^3 \sm \nu L;\La)$.
Now \[H_0(S^3;\sm \nu L;\La) \cong \Z \xrightarrow{\cong} H_0(D^4 \sm \nu L;\La) \cong \Z\] is an isomorphism. Combined with $H_1(D^4\sm \nu P;\La)=0$ and the long exact sequence of the pair we deduce that $H_i(D^4 \sm \nu P,S^3 \sm \nu L;\La)=0$ for $i=0,1$.  Then similarly to the proof of Lemma~\ref{lemma:general-homology}, the UCSS and ~\cite[Lemma~2.1]{Borodzik-Friedl-2013-alg-unknotting} imply that $H_2(D^4 \sm \nu P,P \times S^1;\La)$ is a free module.  The rank must be $r-1+2g$ since we know that this is the rank over $\La_S$, so indeed $H_2(D^4 \sm \nu P,P \times S^1;\La) \cong \La^{r-1+2g}$ as claimed.

We have now seen that the exact sequence \eqref{eqn:PxS1} is equivalent to
\[0 \to \La^{r-1+2g} \to \La^{r-1+2g} \to \Z^{r-1+2g} \to 0.\]
Here we consider $\Z$ as a $\La$-module where $t$ acts as the identity.
Representing generators of $\Z^{r-1+2g}$ by curves $\gamma_k \times \{\pt\}$ in $P \times S^1$, we can lift them to a basis of $\La^{2g+1-r}$, by extending them to elements of $H_2(D^4 \sm \nu P, P \times S^1;\La)$. That is, choose a $\Z$-trivial surface $D_k \looparrowright D^4 \sm \nu P$ with boundary $\gamma_k \times \{\pt\}$, for each $k=1,\dots,2g+1-r$, as in Construction~\ref{construction:spheres-D4-sm-P}.

The surfaces $\S_k$ from Construction~\ref{construction:spheres-D4-sm-P} satisfy $[\S_k] = (t-1)\cdot [D_k] \in H_2(D^4 \sm \nu P,P \times S^1;\La)$. Therefore the $\{\S_k\}$ also represent a basis for $H_2(D^4 \sm \nu P,P \times S^1;\La)$ over $\La_S$.  Since the $\{\S_k\}$ are closed surfaces they lift to $H_2(D^4 \sm \nu P;\La_S)$. Then by \eqref{eqn:La_s-iso} it follows that the $\{\S_k\}$ also represent a basis for $H_2(D^4 \sm \nu P;\La_S) \cong \La_S^{r-1+2g}$.
\end{proof}

Recall that $W_P := D^4 \sm \nu P \cup_{P \times S^1} (G \times S^1)$.

\begin{lemma}
  The inclusion induced map $H_i(D^4\sm \nu P;\La_S) \to H_i(W_P;\La_S)$ is an isomorphism for every $i$.
  The nonvanishing homology groups of $W_P$ are as follows.
	\[H_i(W_P;\La) \cong \begin{cases} \Z & i=0 \\ \La^{2g+ r-1} & i=2. \end{cases}\]
Over $\La_S$, a basis for $H_2(W_P;\La_S)$ is given by the collection of immersed surfaces $\{\S_k\}_{k=1}^{r-1+2g}$.
\end{lemma}

\begin{proof}
Lemma~\ref{lemma:general-homology} tells us the homology of $W_P$ with $\La$ coefficients, except for the rank of $H_2(W_P;\La)$.
For $Q \in \{P,G\}$ we have $H_i(Q \times S^1;\La) \cong H_i(Q \times \R;\Z) \cong H_i(Q;\Z)$ for every $i$. Therefore $H_i(Q\times S^1;\La)$ is annihilated by $t-1$ and so
\[H_i(Q \times S^1;\La_S) \cong H_i(Q \times S^1;\La) \otimes_{\La} \La_S = 0.\]
It then follows from the Mayer-Vietoris sequence for homology with $\La_S$ coefficients, using the decomposition  $W_P := D^4 \sm \nu P \cup_{P \times S^1} (G \times S^1)$, that $H_i(D^4\sm \nu P;\La_S) \xrightarrow{\cong} H_i(W_P;\La_S)$ is an isomorphism for all~$i$. In particular this implies that $H_2(W_P;\La_S) \cong H_2(D^4 \sm \nu P;\La_S) \cong \La_S^{r-1+2g}$, so indeed $H_2(W_P;\La) \cong \La^{r-1+2g}$ as claimed.
The fact that the inclusion induced map is an isomorphism over $\La_S$ implies that the same immersed surfaces $\{\S_k\}_{k=1}^{r-1+2g}$ for $H_2(D^4\sm \nu P;\La)$ in Construction~\ref{construction:spheres-D4-sm-P} and Lemma~\ref{lemma:H2-and-basis-D4-P} also represent a basis for $H_2(W_P;\La_S) = H_2(W_P;\La) \otimes_{\La} \La_S$.
\end{proof}

\begin{construction}\label{construction:surfaces-using-SSs}
We use a $\Z$-trivial surface system (Definition~\ref{defn:ZTS}) to modify the $\S_k$ for $k=1,\dots,r-1$.
  We may suppose that there is a collar $S^3 \times I \subseteq D^4$ and that $P \cap (S^3 \times I)  = L \times I$.  For $j=1,\dots,r-1$, we consider $\gamma_j := L \times \{j/r\} \subseteq L \times I \subseteq P$. Push the Seifert surface $F_j$ for the $j$th component of $L$ to the level $S^3 \times \{j/r\}$.  Now use $F_j$ in place of the immersed surface $D_j$ in Construction~\ref{construction:spheres-D4-sm-P} to surger $\gamma_j \times S^1$ to another embedded surface $\S_{F_j}$. Since $F_j$ is part of a $\Z$-trivial surface system, the embedded surface $\S_{F_j}$ has the property that every curve on it represents the trivial element of $\pi_1(W_P) \cong \Z$.  Therefore $\S_{F_j}$ represents a homology class in $H_2(W_P;\La)$.  Note that this is a special case of the surfaces from Construction~\ref{construction:spheres-D4-sm-P}.

  By using these surfaces that originate from surfaces in $S^3$, we obtain some crucial control on intersections. For $j \neq k$, we have $\S_{F_j} \cap \S_{F_k} = \emptyset$.  Moreover, in the construction of the rest of the surfaces $\S_k$, for $k=r,\dots,r-1+2g$, as in Construction~\ref{construction:spheres-D4-sm-P}, we may assume without loss of generality that the surfaces $D_k$ are disjoint from the collar $S^3 \times I$. It follows that $\S_{F_j} \cap \S_k = \emptyset$ for every $j \in \{1,\dots,r-1\}$ and for every $k \in \{r,\dots,r-1+2g\}$.
\end{construction}

\begin{lemma}
  $\lambda(\S_{F_i},\S_{F_i})=0$.
\end{lemma}

\begin{proof}
  The torus in the construction of $\S_{F_i}$ can be pushed off itself. Since the $F_i$ induce the zero framing, this can be extended to two disjoint parallel copies of $F_i$ in $S^3 \times \{i/r\}$.
\end{proof}

Now we use these surfaces to compute the intersection form.

\begin{lemma}\label{lemma:int-form-form}
  The intersection form $\lambda \colon H_2(W_P;\La) \times H_2(W_P;\La) \to \La$ can be represented by a matrix of the form
  \[\begin{pmatrix}
    0_{(r-1) \times (r-1)} & 0_{(r-1) \times 2g} \\ 0_{2g \times (r-1)} & A_{2g \times 2g}
  \end{pmatrix}\]
for some matrix $A$ over $\La$ such that $A(1)$ has signature $0$ and $\det (A(1)) \neq 0$.
\end{lemma}

\begin{proof}
 Let \[\mathcal{S}:= \{\S_{F_i}\}_{i=1}^{r-1} \cup \{\S_j\}_{j=1}^{2g}.\]  Let $\{e_i\}_{i=1}^{r-1+2g}$ be a basis for the homology $H_2(W_P;\La) \cong \La^{r-1+2g}$, and suppose that, for integers $y_i$, we have that $(t-1)^{y_i} e_i = \S_{F_i}$ for $i=1,\dots,r-1$, and that $(t-1)^{y_i} e_i = \S_{i}$ for $i=r,\dots,r-1+2g$.  We may make this supposition since we know that $\mathcal{S}$ represents a basis for $H_2(W_P;\La_S) \cong H_2(W_P;\La) \otimes_{\La} \La_S$.

 Since for every $i \in \{1,\dots,r-1\}$, we have that $\S_{F_i}$ is disjoint from all the other surfaces in $\mathcal{S}$, it follows that for $i=1,\dots,r-1$ and for every $j\in \{1,\dots,r-1+2g\}$ we have
 \begin{align*}
   0 & = \lambda((t-1)^{y_i} e_i, (t-1)^{y_j}e_j) = (t-1)^{y_i}\lambda(e_i,e_j)(t^{-1}-1)^{y_j} \\  &= (t-1)^{y_i}\lambda(e_i,e_j)(-t)^{-y_j}(t-1)^{y_j} = (-t)^{-y_j}(t-1)^{y_i+y_j}\lambda(e_i,e_j).
 \end{align*}
  Since $\La$ is an integral domain, it follows that $\lambda(e_i,e_j)=0$.  The matrix representing $\lambda$ is therefore of the form claimed.

To see that the matrix $A(1)$ has nonzero determinant, we consider the long exact sequence
\[H_2(W_P;\Q) \to H_2(W_P,M_L;\Q) \to H_1(M_L;\Q) \to H_1(W_P;\Q)\]
which reduces to
\[\Q^{r-1+2g} \xrightarrow{\lambda_\Q} \Q^{r-1+2g} \to \Q^{r} \to \Q \to 0.\]
The map $\lambda_\Q$ can be represented by a matrix for the ordinary $\Q$-valued intersection form of $W_P$, which can in turn be represented by
\[\begin{pmatrix}
    0 & 0 \\
    0 & A(1)
  \end{pmatrix},\]
  because the basis $\mathcal{S}$ descends to a basis for $H_2(W_P;\Q) \cong H_2(W_P;\La) \otimes_{\La} \Q$, where this isomorphism follows from the UCSS for homology.
Since $\ker (\Q^r \to \Q) \cong \Q^{r-1}$, it follows by exactness that $A(1)$ is nonsingular over $\Q$ and so $\det(A(1)) \neq 0$.

Finally, it was shown in \cite[Proof~of~Lemma~5.4]{Nagel-Powell} that the signature of the intersection form of $W_P$ is zero for links whose pairwise linking numbers are all zero.  The proof there was for a pushed in Seifert surface, but the part that computes the ordinary signature works for any $\Z$-surface.
\end{proof}

Lemma~\ref{lemma:end-pf} below completes the proof of Theorem~\ref{thm:weakly-Z-slice-implies-Alexander-module-condition-bdy-links} \eqref{item:main-1} $\Rightarrow$ \eqref{item:main-2}, which was the last remaining implication we needed to prove.  It uses the following result.

\begin{theorem}\label{lem:general-fact-conway}
  Let $L$ be a boundary link.  The intersection form of a compact, connected, oriented $4$-manifold $W$ with $\partial W \cong M_L$, and the inclusion induced map $\phi \colon \pi_1(\partial W)\to \pi_1(W) \cong \Z$ onto, presents the Blanchfield form on the torsion part $TH_1(M_L;\La_S)$, where the $\La$ coefficients are determined by $\phi$.
\end{theorem}

\begin{proof}
Most proofs of variants of this, such as in \cite{Borodzik-Friedl-2015-unknotting-number-and-classical-invariants}, assume that $H_1(\partial W;\La)$ is $\La$-torsion.
However Conway~\cite{Conway} works with link exteriors, and shows how to compute the Blanchfield pairing on $H_1(X_L;\La_S)$ in terms of a totally connected $C$-complex, by computing the intersection pairing of the complement $W$ of the $C$-complex pushed in to $D^4$, and relating the homology of $\partial W$ to the homology of $X_L$.   But we have $H_1(X_L;\La_S) \cong H_1(M_L;\La_S)$ by Lemma~\ref{lem:XL-vs-ML}.   In the proof, the only property that Conway uses of the complement $W$ of the totally connected $C$-complex is that $H_1(W;\La)=0$, and $\pi_1(\partial W)\to \pi_1(W) \cong \Z$ is onto. So in fact his proof also proves the statement we want, for a more general $4$-manifold with boundary $M_L$.
\end{proof}

\begin{lemma}\label{lemma:end-pf}
  The rank of $H_1(M_L;\La)$ is $r-1$, the Blanchfield form on  $TH_1(M_L;\La)$ is presented by the Hermitian matrix~$A$, which is of size $2g$ and has $\sigma(A(1))=0$.
\end{lemma}

\begin{proof}
By Lemma~\ref{lemma:int-form-form} and Lemma~\ref{lemma:algebraic}, we deduce that $H_1(M_L;\La) \cong \La^{r-1} \oplus TH_1(M_L;\La)$, where $TH_1(M_L;\La)$ satisfies $\ord TH_1(M_L;\La)(1) = \pm 1$ and is presented by $A$, where  \[\begin{pmatrix}
   0_{(r-1) \times (r-1)} & 0_{(r-1) \times 2g} \\ 0_{2g \times (r-1)} & A_{2g \times 2g}
  \end{pmatrix}\]
  represents the intersection form over $\La$ of the compact, oriented $4$-manifold $W_P$, whose boundary is $M_L$ and with $\pi_1(W_P) \cong \Z$ and $\pi_1(M_L) \to \pi_1(W_P)$ onto.
 It therefore follows from Theorem~\ref{lem:general-fact-conway} that $A$, which is of size $2g$, presents the Blanchfield form on $TH_1(M_L;\La_S)$.
 By Lemma~\ref{lem:XL-vs-ML}, $TH_1(M_L;\La) \cong TH_1(X_L;\La)$, and by Lemma~\ref{lem:algboundarylinkchar} we know that $\ord TH_1(X_L;\La)(1) = \pm 1$. Therefore $\ord TH_1(M_L;\La)(1) = \pm 1$.  As in the proof of Theorem~\ref{thm:2=>1}, this implies that multiplication by $t-1$ induces an isomorphism on $TH_1(M_L;\La)$, so in fact $A$ computes the Blanchfield form over $\La$ as well.
\end{proof}

\section{Applications}\label{sec:proofs-applications}

In this section we prove the applications stated in the introduction. First we recall the notion of the algebraic genus of a knot, and present a corollary to Theorem~\ref{thm:main} about it.

The \emph{algebraic genus} of a knot $K$, denoted by $\ga(K)$, is defined by
\[\ga(K):=\min
\left\{m-n\ \middle| \begin{array}{l}\text{$K$ admits an $2m\times 2m$ Seifert matrix of the form }\begin{pmatrix}
A & * \\
* & *
\end{pmatrix},\\
\text{where $A$ is a } 2n\times 2n \text{ submatrix with }\det(tA-A^T)=t^n\end{array}\right\}.\]
It was proven in \cite[Corollary 1.5]{FellerLewark_19} that $g_\Z(K) = \ga(K),$ and moreover~\cite[Theorem 1.1]{FellerLewark_19} that $\ga(K)$ is equal to the minimal $g$ for which the Blanchfield pairing of $K$ can be presented by a size $2g$ Hermitian matrix $A$ over $\Lambda$ with the signature of $A(1)$ zero. Using the above terminology, we can prove Corollary~\ref{cor:bandsum}. We restate the corollary.

\begin{corollary}\label{cor:gzisgalg}Let $L$ be an $r$-component boundary link and let $K_L$, $K_L'$ be knots, both of which are obtained by performing $r-1$ internal band sums on $L$. Furthermore, suppose that internal bands for $K_L$ are performed disjoint from some collection of disjoint Seifert surfaces for $L$. Then $$g_\Z(L) = g_\Z(K_L) \leq g_\Z(K_L').$$
\end{corollary}

\begin{proof}
As in the proof of Theorem~\ref{thm:main}~\eqref{item:main-3} $\Rightarrow$ \eqref{item:main-1}, if $K_L'$ bounds a $\Z$-surface $\Sigma$ of genus $g$, then $L$ also bounds a $\Z$-surface of genus~$g$ obtained by gluing the genus $0$ cobordism from $L$ to $K_L'$ with $\Sigma$. Hence $g_\Z(L) \leq g_\Z(K_L')$ and similarly $g_\Z(L) \leq g_\Z(K_L)$.

By Theorem~\ref{thm:main}, if $L$ bounds a $\Z$-surface of genus $g$, then the torsion part of the Blanchfield form of $M_L$ is presented by a size $2g$ Hermitian square matrix $A$ over $\La$ with the signature of $A(1)$ zero. Moreover, by Theorem~\ref{thm:2=>1}, $A$ also presents the Alexander module of~$K_L$. This implies that $g_\Z(K_L) \leq g_\Z(L)$ and concludes the proof.
\end{proof}

\subsection{Shake genus and generalised cabling}

We start with a reformulation of  the $\Z$-shake genus of a knot $K$, denoted by $\gs(K)$.
Recall that the $\Z$-shake genus of a knot $K$ is the minimal genus of a surface $\Sigma$ representing a generator of $H_2(X_0(K);\Z)$ with $\pi_1(X_0(K) \sm \Sigma) \cong \Z$, generated by a meridian of $\Sigma$.
Also recall that $P_{p,n}(K)$ denotes the generalisation of a cable link obtained by $p+n$ parallel copies of $K$ with pairwise vanishing linking numbers, where $p$-components have the same orientation as $K$ and the remaining $n$-components have the opposite orientation.

\begin{lemma}\label{lem:shake}
 The $\Z$-shake genus of $K$ satisfies
\begin{equation*}
\gs(K)=\min \left\{g_\Z(P_{\ell+1,\ell}(K)) \mid \ell \in \mathbb{N}_0\right\}.
\end{equation*}
\end{lemma}

\begin{proof}
First we show that $\gs(K) \geq \min \left\{g_\Z(P_{\ell+1,\ell}(K)) \mid \ell \in \mathbb{N}_0\right\}$. Let $S$ be a locally flat surface of genus $g$ in the $0$-trace, denoted by $X_0(K)$, representing a generator of $H_2(X_0(K);\Z) \cong \Z$, such that $\pi_1(X_0(K)\setminus \nu S)\cong\Z$.
To prove this we construct a $4$-manifold $W$ with $\partial W=M_{P_{k+1,k}(K)}$ for some $k$, using the same construction as in Section~\ref{sec:1=>2}. Moreover $W$ will satisfy that $\pi_1(W) \cong \Z$ and $H_2(W;\La) \cong \La^{2g + 2k}$. The proof of Theorem~\ref{thm:weakly-Z-slice-implies-Alexander-module-condition-bdy-links} will imply that there is a size $2g$ Hermitian square matrix $A$ over $\La$ such that  $A(1)$ has signature $0$ and such that $A$ presents the Blanchfield form of $M_{P_{k+1,k}(K)}$ on $TH_1(M_{P_{k+1,k}(K)};\La)$.  It will then follow that
\[\min \left\{g_\Z(P_{\ell+1,\ell}(K)) \mid \ell \in \Z\right\} \leq g_\Z(P_{k+1,k}(K)) \leq g.\]

 To achieve this, make $S$ transverse to the cocore of the 2-handle, and remove a neighbourhood $N$ of the cocore.
This yields a 4-manifold homeomorphic to $D^4$ with the link $P_{k+1,k}(K) = \partial N\cap S$ in $\partial D^4 = S^3$, for some $k$, extending to a locally flat genus $g$ surface in $D^4$. Since we removed a disjoint union of discs $N\cap S$, the result is a connected genus $g$ surface $P$.

Apply the Seifert-van Kampen theorem to the union
\[X_0(K)\setminus \nu S= (D^4\setminus \nu P)\cup (N\setminus \nu S),\]
where the union is over the complement in $S^1 \times D^2$ of $2k +1$ parallel copies of the core $S^1 \times \{0\}$.
This yields a push out
 \[\xymatrix{F_{2k +1} \times \Z \ar[r] \ar[d] & \pi_1(D^4 \sm \nu P) \ar[d] \\
 F_{2k+1} \ar[r] & \Z,}\] where $F_{2k +1}$ is the free group on $2k+1$ generators.  The $\Z$ in the top left corner is generated by a zero-framed longitude of $K$, while the $\Z$ in the bottom right corner is generated by a meridian of~$P$.
Hence we have that $\pi_1(D^4 \sm \nu P) / \langle \langle \lambda \rangle \rangle \cong \Z$, where $\lambda$ represents a longitude of $K$.  Since $\lambda$ lies in the second derived subgroup, it follows that the commutator subgroup, or first derived subgroup, equals the second derived subgroup, and is therefore perfect.

 Let $G$ be a handlebody with $\partial G = P \cup_{\partial P} \bigcup_{i=1}^{2k+1} D^2$, a closed surface of genus $g$, and then define
\[W = W_P := D^4 \sm \nu P \cup_{P \times S^1} (G \times S^1),\]
 as in Section~\ref{sec:1=>2}.
 As before, choose the framing of the normal bundle of $P$ so that every simple closed curve on $P \times \{\pt\} \subset P \times S^1$ lies in the commutator subgroup of $\pi_1(D^4 \sm \nu P)$.  Since gluing on $G \times S^1$ kills the longitude of each component of $P_{k+1,k}(K)$, it follows that
 \[\pi_1(W) \cong \pi_1(D^4 \sm \nu P) / \langle \langle \lambda,\gamma_1,\dots,\gamma_g \rangle \rangle\]
 where $\lambda$ represents the longitude of $K$ as above, and $\gamma_1,\dots,\gamma_k \subseteq P \times \{\pt\}$ are push offs to the normal circle bundle of curves on $P$ that generate $\ker(H_1(\partial G;\Z) \to H_1(G;\Z))$.  Since we know that $\pi_1(D^4 \sm \nu P) / \langle \langle \lambda \rangle \rangle \cong \Z$, and the $\gamma_i$ lie in the commutator subgroup of $\pi_1(D^4 \sm \nu P)$, we deduce that $\pi_1(W) \cong \Z$.

  In the proof of Theorem~\ref{thm:weakly-Z-slice-implies-Alexander-module-condition-bdy-links}, as noted in Remark~\ref{remark:what-we-use}, it was sufficient to assume that $\pi_1(D^4 \sm \nu P)$ is of the form $\Z \ltimes \Gamma$, where the commutator subgroup $\Gamma$ is perfect, and that $\pi_1(W_P) \cong \Z$.  The proof of Theorem~\ref{thm:weakly-Z-slice-implies-Alexander-module-condition-bdy-links} then implies that there is a size $2g$ Hermitian square matrix $A$ over $\La$ of the required form as in Theorem~\ref{thm:main}~\eqref{item:main-2}. Thence Theorem~\ref{thm:main} implies that $g_\Z(P_{k+1,k}(K)) \leq g$, as required.

For the other inequality, cap off a $\Z$-surface for $P_{\ell+1,\ell}(K)$ with $2\ell +1$ appropriately oriented parallel copies of the core of the 2-handle of $X_0(K)$, to construct a $\Z$-shake surface of genus $g$ in~$X_0(K)$.
\end{proof}

As a tangent, and to point to a subtlety which necessitates the above proof, we ask the following questions.
Does there exist a locally flat surface $P$ in $D^4$, with boundary $P_{\ell+1,\ell}(K) \subseteq S^3$ for some $K,\ell$, such that
$\pi_1(D^4\setminus \nu P)/\langle\langle\lambda\rangle\rangle\cong \Z$, where $\lambda$ represents a longitude of $K$, but for which $\pi_1(D^4\setminus \nu P)$ is not cyclic? A negative answer to this question would imply that every surface whose complement has cyclic fundamental group, representing a generator of $H_2(X_0(K);\Z)$, is isotopic to the union of parallel copies of the core of the 2-handle and a $\Z$-surface in $D^4$.

\begin{figure}[h]
\includegraphics[width=.6\textwidth]{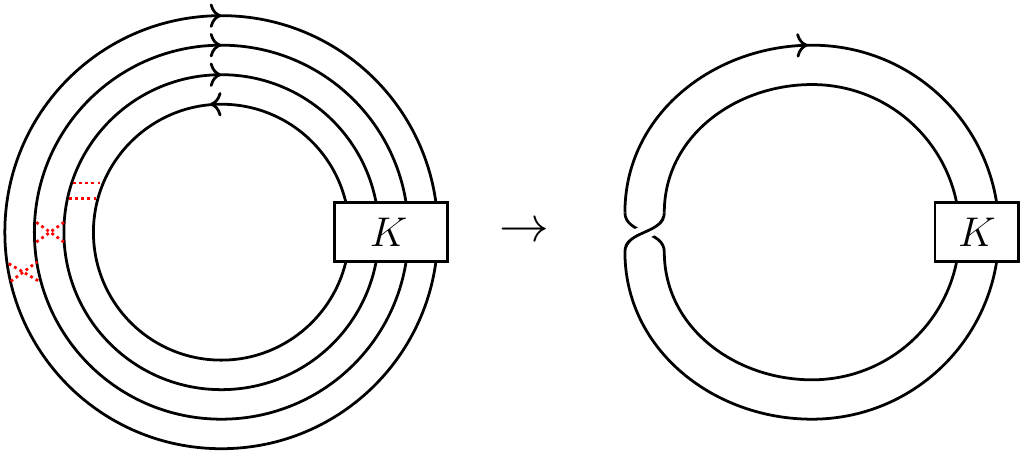}
\caption{Performing $3$ internal band sums on $P_{3,1}(K)$ yields the cabled knot $K_{2,1}$. Solid boxes indicate that all the strands passing through the box are tied into 0-framed parallels of the knot $K$.}\label{fig:shakeg}
\end{figure}

Using Corollary~\ref{cor:gzisgalg}, we prove Theorem~\ref{cor:shake} and Corollary~\ref{cor:parallel}. We prove them together, since the proofs are similar. We recall the statements.

\begin{corollary}\leavevmode
\begin{enumerate}[label=(\emph{\roman*}),font=\upshape]
\item\label{item:shake-i}   For every knot $K$, the $\Z$-genus of $K$ equals the $\Z$-shake genus of $K$.
\item\label{item:shake-ii} Let $p$ and $n$ be integers and let $w=p-n$. If $w=0$, then $P_{p,n}(K)$ is 
$\Z$-weakly slice, and otherwise
$$g_\Z(P_{p,n}(K))=g_\Z(K_{w,1})=g_\Z(K_{w,-1}) \leq g_\Z(K).$$
\end{enumerate}
\end{corollary}

\begin{proof}
We prove~\ref{item:shake-ii} first.
Note that $P_{p,n}(K)$ is a $(p+n)$-component boundary link, and $(p+n)$ disjointly embedded Seifert surfaces are obtained by taking parallel copies of a Seifert surface for $K$ with appropriate orientations. Let $w=p-n$.
Perform $w-1=p+n-1$ internal band sums on~$L$ to obtain the unknot if $w=0$, and the knot $K_{w,1}$ if $w \neq 0$ (see Figure~\ref{fig:shakeg}).
Similarly, we can also construct $K_{w,-1}$ by performing $w=p+n-1$ internal band sums on~$L$.
It follows that $P_{p,n}(K)$ is $\Z$-weakly slice if $w=0$.
For $w \neq 0$, by Corollary~\ref{cor:gzisgalg}, in particular the equality $g_\Z(L) = g_{\Z}(K)$, where the knot $K$ is obtained from the $r$-component link $K$ by $r-1$ internal band sums away from a collection of disjoint Seifert surfaces for~$L$,
we conclude that $g_\Z(P_{p,n}(K))=g_\Z(K_{w,1})=g_\Z(K_{w,-1})$. The fact that $g_\Z(K_{w,1}) \leq g_\Z(K)$ follows from \cite[Theorem 1.2]{Feller-Miller-Pinzon-Caicedo:2019-1} and \cite[Theorem 4]{McCoy:2019-1}.  This completes the proof of Corollary~\ref{cor:parallel}.

Now we prove~\ref{item:shake-i}, which is Theorem~\ref{cor:shake}. By Lemma~\ref{lem:shake} and~\ref{item:shake-ii}, we have $$\gs(K)=\min \left\{g_\Z(P_{\ell+1,\ell}(K)) \mid \ell \in \Z\right\} = g_\Z(P_{1,1}(K)) = g_\Z(K).$$ This completes the proof of Theorem~\ref{cor:shake}.
\end{proof}

\subsection{Good boundary links}

Next we prove Corollary~\ref{cor:goodboundary}, whose statement we recall.

\begin{corollary}
Every good boundary link is $\Z$-weakly slice.
\end{corollary}

\begin{proof}
By Corollary~\ref{cor:weakly-slice-torsion-free}, every boundary link $L$ with $TH_1(M_L;\La)=0$ is 
$\Z$-weakly slice.
  Let $L$ be an $r$-component good boundary link and let $F$ be the free group on $r$ generators.
  Then $H_1(M_L;\Z F) = 0$. We apply the universal coefficient spectral sequence~\cite[Theorem~10.90]{Rotman:2009} with $E_2$ page $\Tor_p^{\Z F}(H_q(M_L;\Z F),\La)$, computing $H_{p+q}(M_L;\La)$. Since \[H_1(M_L;\Z F) \otimes_{\Z F} \La \cong \Tor_0^{\Z F}(H_1(M_L;\Z F),\La) =0,\] we have:
  \[H_1(M_L;\La) \cong \Tor_1^{\Z F}(H_0(M_L;\Z F),\La) \cong H_1(\vee^r S^1;\La) \cong \La^{r-1}.\]
 Since $\La^{r-1}$ is free, $TH_1(M_L;\La)=0$, so $L$ is $\Z$-weakly slice by Corollary~\ref{cor:weakly-slice-torsion-free}.
\end{proof}

\subsection{Whitehead doubles}

Finally we prove Corollary~\ref{cor:Wh}. Here is the statement.

\begin{corollary}\label{cor:Wh-main-text}
If $L=L_1 \cup L_2$ is a $2$-component link, then
\[g_\Z(\Wh(L))=\left\{
\begin{array}{cl}
0&\text{if $\lk(L_1,L_2)=0$},\\
1&\text{otherwise.}\end{array}\right.\]
Moreover, if $L=L_1 \cup L_2 \cup L_3$ is a $3$-component link, then $\Wh(L)$ is 
$\Z$-weakly slice if and only if either $($i$)$ $L$ has vanishing linking numbers, or $($ii$)$ for some $i,j,k$ with $\{i,j,k\}=\{1,2,3\}$:
\begin{enumerate}[label=(\emph{\alph*}),font=\upshape]
  \item the signs of the clasps of $\Wh(L_i)$ and $\Wh(L_j)$ disagree,
  \item $\lk(L_i,L_j)=0$,  and
  \item $|\lk(L_i,L_k)|=|\lk(L_j,L_k)|.$
\end{enumerate}
\end{corollary}

\begin{proof}
 Suppose $L=L_1 \cup L_2$ is a $2$-component links with $\lk(L_1,L_2)=n$. Let $F_1, F_2$ be the standard disjoint genus 1 Seifert surfaces for $\Wh(L)$. Then by performing an internal band sum, where the band does not intersect $F_1$ and $F_2$, we get a knot $K$ with a genus two Seifert surface and Seifert matrix
\[M=\begin{pmatrix}
    0 & a_1 & n& n \\
    0 & 0 & n& n \\
    n & n & 0& a_2 \\
    n & n & 0& 0 \\
  \end{pmatrix},\]
where $a_1, a_2 \in \{1, -1\}$. By Corollary~\ref{cor:gzisgalg}, we have $g_\Z(\Wh(L)) = \ga(K).$
The computation
$$\det(tM-M^T)=\sum_{i=0}^{2} c_i t^i + (4 a_1 a_2 n^2)\cdot t^3 +(-a_1a_2n^2)\cdot t^4$$
implies that $g_\Z(\Wh(L))=0$ if and only if $n=0$. Moreover, note that there is a $2\times 2$ submatrix
$$A= \begin{pmatrix}
    0 & a_1 \\
    0 & 0 \end{pmatrix} \text{ so that } \det(tA-A^T) = t.$$
    Hence, if $n\neq 0$, then $g_\Z(\Wh(L))=1$.

    Now, suppose $L=L_1 \cup L_2 \cup L_3$ is a $3$-component links with $\lk(L_1,L_2)=n_3, \lk(L_1,L_3)=n_2$, and $\lk(L_1,L_3)=n_1$. Again, performing two internal band sums, where the bands do not intersect the standard  disjoint Seifert surfaces, we obtain a knot $K$ with a genus three Seifert surface and a Seifert matrix
\[M=\begin{pmatrix}
    0 & a_1 & n_3& n_3& n_2& n_2 \\
    0 & 0 & n_3& n_3& n_2& n_2 \\
    n_3 & n_3 & 0& a_2& n_1& n_1  \\
    n_3 & n_3 & 0& 0& n_1& n_1  \\
    n_2 & n_2 & n_1& n_1& 0& a_3  \\
    n_2 & n_2 & n_1& n_1& 0& 0
  \end{pmatrix},\]
where $a_1, a_2, a_3 \in \{1, -1\}$. Again, we have $g_\Z(\Wh(L)) = \ga(K)$.

A straightforward computation yields that
$$\det(tM-M^T)=\sum_{i=0}^{4} c_it^i+(12 n_1n_2n_3a_1a_2a_3- n_1^2a_2a_3 - n_2^2 a_1a_3 - n_3^2a_1a_2)\cdot t^5+(-2n_1n_2n_3 a_1a_2a_3)\cdot t^6.$$ Note that $g_\Z(\Wh(L))=\ga(K)=0$ if and only if $\det(tM-M^T)=t^3$, and that this implies either $n_1=n_2=n_3=0$ or
$$a_i=-a_j, \quad n_k=0, \quad\text{and}\quad |n_i|=|n_j| \quad\text{for}\quad \{i,j,k\}=\{1,2,3\}.$$
Furthermore, it can be easily verified that the above assumptions imply that $\det(tM-M^T)=t^3$. This completes the proof of Corollary~\ref{cor:Wh}.
\end{proof}

\bibliographystyle{alpha}
\bibliography{research}
\end{document}